\documentclass[a4paper, 11pt]{article}

\usepackage{a4wide}
\usepackage[english]{babel}
\usepackage{amsmath}
\usepackage{psfrag}
\usepackage[latin1]{inputenc}
\usepackage[T1]{fontenc}
\usepackage{amsthm} 
\usepackage{amssymb,amsxtra}
\usepackage[usenames]{color}
\usepackage{stmaryrd}
\usepackage{mathrsfs}

\newcommand{\NN}{\mathbb{N}}

\newcommand{\w}{\wedge}
\newcommand{\he}{\hat{e}}
\newcommand{\hk}{\hat{k}}
\newcommand{\hl}{\hat{l}}

\DeclareMathOperator{\res}{Res}
\DeclareMathOperator{\ind}{Ind}

\DeclareMathOperator{\Rad}{Rad}
\DeclareMathOperator{\Hd}{Hd}

\DeclareMathOperator{\supp}{supp}
\DeclareMathOperator{\Tr}{Tr}

\renewcommand{\leq}{\leqslant}
\renewcommand{\geq}{\geqslant}
\renewcommand{\unlhd}{\trianglelefteqslant}

\begin{document}

\swapnumbers
\theoremstyle{definition}
\newtheorem{defi}{Definition}[section]
\newtheorem{rem}[defi]{Remark}
\newtheorem{ques}[defi]{Question}
\newtheorem{expl}[defi]{Example}
\newtheorem{conj}[defi]{Conjecture}
\newtheorem{claim}[defi]{Claim}
\newtheorem{nota}[defi]{Notation}
\newtheorem{noth}[defi]{}

\theoremstyle{plain}
\newtheorem{prop}[defi]{Proposition}
\newtheorem{lemma}[defi]{Lemma}
\newtheorem{cor}[defi]{Corollary}
\newtheorem{thm}[defi]{Theorem}

\renewcommand{\proofname}{\textsl{\textbf{Proof}}}

\baselineskip=14pt

\begin{center}
{\bf\Large Vertices of Simple Modules of Symmetric Groups \\ Labelled by Hook Partitions}

\medskip
 
Susanne Danz and Eugenio Giannelli 

\today

\begin{abstract}
\noindent
In this article we study the vertices of simple modules for the symmetric groups in prime characteristic $p$. In particular, we complete the classification of the vertices of simple $F\mathfrak{S}_n$-modules labelled by hook partitions.

\smallskip

\noindent
{\bf Mathematics Subject Classification (2010):} 20C20, 20C30.

\noindent
{\bf Keywords:} symmetric group, simple module, hook partition, vertex.

\end{abstract}
\end{center}

\section{Introduction}\label{sec intro}

Introduced by J.~A.~Green in 1959 \cite{Gr}, {\sl vertices} of indecomposable modules over modular
group algebras have proved to be important invariants linking the global and local
representation theory of finite groups over fields of positive characteristic. 
Given a finite group $G$ and a field $F$ of characteristic $p>0$,
by Green's result, the
vertices of every indecomposable $FG$-module form a $G$-conjugacy class of $p$-subgroups of $G$.
Moreover, vertices of simple $FG$-modules
are known to satisfy a number of very restrictive properties, most notably in consequence of Kn\"orr's Theorem  \cite{Kn}.
The latter, in particular, implies that vertices of simple $FG$-modules in blocks with abelian defect groups
have precisely these defect groups as their vertices. Despite this result, the precise structure of
vertices of simple $FG$-modules is still poorly understood, even for very concrete groups and modules.

\smallskip

The aim of this paper is to complete the description of the vertices of a distinguished class of simple modules
of finite symmetric groups. 
Throughout,  let $n\in\NN$, and let $\mathfrak{S}_n$ be the symmetric group of degree $n$.
 Then, as is well known, the isomorphism classes of simple $F\mathfrak{S}_n$-modules
are labelled by the $p$-regular partitions of $n$.  We denote the simple $F\mathfrak{S}_n$-module corresponding to
a $p$-regular  partition $\lambda$ by $D^\lambda$. If $\lambda=(n-r,1^r)$, for some $r\in\{0,\ldots,p-1\}$, then
$\lambda$ is called a $p$-regular {\sl hook partition} of $n$. Whilst, in general, even the dimensions of the  simple 
$F\mathfrak{S}_n$-modules are unknown, one has a neat description of an $F$-basis of $D^{(n-r,1^r)}$; we shall
comment on this in \ref{noth exterior} below.

The problem of determining the vertices of the simple $F\mathfrak{S}_n$-module $D^{(n-r,1^r)}$ has been 
studied before by Wildon in \cite{W}, by M\"uller and Zimmermann in \cite{MZ}, and by the first author in \cite{D}.
In consequence of these results, the vertices of $D^{(n-r,1^r)}$ have been known, except in the case
where $p>2$, $r=p-1$ and $n\equiv p\pmod{p^2}$. In Section~\ref{sec proof} of the current paper we shall now prove the
following theorem, which together with \cite[Corollary~5.5]{D} proves \cite[Conjecture~1.6(a)]{MZ}.

\begin{thm}\label{thm main}
Let $p>2$, let $F$ be a field of characteristic $p$, and let $n\in\NN$ be such  that $n\equiv p\pmod{p^2}$. Then
the vertices of the simple $F\mathfrak{S}_n$-module $D^{(n-p+1,1^{p-1})}$ are precisely the Sylow $p$-subgroups
of $\mathfrak{S}_n$.
\end{thm}

Our key ingredients  for proving Theorem~\ref{thm main} will be the Brauer construction in the sense of Brou\'e \cite{Br}
and Wildon's result in \cite{W}. Both of these will enable us to obtain lower bounds on the vertices of $D^{(n-p+1,1^{p-1})}$,
which together will then provide sufficient information to deduce Theorem~\ref{thm main}.

To summarize, the abovementioned results in \cite{D,MZ,W} and Theorem~\ref{thm main} 
lead to the following exhaustive description of the vertices of the modules $D^{(n-r,1^r)}$:

\begin{thm}\label{thm vertices}
Let $F$ be a field of characteristic $p>0$, and let $n\in\NN$. Let further $r\in\{0,1\ldots,p-1\}$, and
let $Q$ be a vertex of the simple $F\mathfrak{S}_n$-module $D^{(n-r,1^r)}$.

\smallskip

{\rm (a)}\, If $p\nmid n$ then $Q$ is $\mathfrak{S}_n$-conjugate to a Sylow $p$-subgroup
of $\mathfrak{S}_{n-r-1}\times \mathfrak{S}_r$.

\smallskip

{\rm (b)}\, If $p=2$, $p\mid n$ and $(n,r)\neq (4,1)$ then $Q$ is a Sylow $2$-subgroup 
of $\mathfrak{S}_n$.

\smallskip

{\rm (c)}\, If $p=2$, $n=4$ and $r=1$ then $Q$ is the unique Sylow $2$-subgroup
of $\mathfrak{A}_4$.

\smallskip

{\rm (d)}\, If $p>2$ and $p\mid n$ then $Q$ is a Sylow $p$-subgroup of $\mathfrak{S}_n$.
\end{thm}

In the case where $p\nmid n$,
the simple module $D^{(n-r,1^r)}$ is isomorphic to the Specht $F\mathfrak{S}_n$-module $S^{(n-r,1^r)}$, by work of Peel \cite{P}.
Thus assertion~(a) follows immediately from \cite[Theorem 2]{W}.
Assertions (b) and (c) have been established by M\"uller and Zimmermann \cite[Theorem~1.4]{MZ}.
Moreover, if $p>2$, $p\mid n$ and $r<p-1$ then assertion (d) can also be found in \cite[Theorem~1.2]{MZ}.
The case where $p>2$, $p\mid n$, $r=p-1$ was treated in \cite[Corollary~5.5]{D}, except when
$n\equiv p\pmod{p^2}$, which is covered by Theorem~\ref{thm main} above. 

\medskip

We should also like to comment on the sources of the simple
$F\mathfrak{S}_n$-modules $D^{(n-r,1^r)}$.
For $r=0$, we get the trivial $F\mathfrak{S}_n$-module $D^{(n)}$, which has of course trivial source.
 If $p\mid n$, then the module $D^{(n-1,1)}$
restricts indecomposably to its vertices, by \cite[Theorems~1.3, 1.5]{MZ}, except when $p=2$ and $n=4$.
For $p=2$, the simple $F\mathfrak{S}_4$-module $D^{(3,1)}$ has trivial source, by  \cite[Theorem~1.5]{MZ}.
If $p\nmid n$ then $D^{(n-r,1^r)}\cong S^{(n-r,1^r)}$ has always trivial sources; see, for instance
 \cite[Theorem~1.3]{MZ}. However, in the case where $p>2$, $p\mid n$ and $r>1$, we do not know
 the sources of $D^{(n-r,1^r)}$. In these latter cases,  the restrictions of $D^{(n-r,1^r)}$
 to its vertices should, conjecturally, be indecomposable, hence should be sources of $D^{(n-r,1^r)}$; see \cite[Conjecture~1.6(b)]{MZ}.
This conjecture has been verified computationally in several cases, see \cite{D,MZ}, but remains still open in general.

\bigskip

\noindent
{\bf Acknowledgements:} The first author has been supported through DFG Priority Programme `Representation Theory' (Grant  \# DA1115/3-1),
and a Marie Curie Career Integration Grant (PCIG10-GA-2011-303774). The results of this article were achieved during the visit of the second author to the 
University of Kaiserslautern. He gratefully acknowledges his PhD supervisor Dr.~Mark Wildon for supporting the visit. 
He also thanks the research group \textsl{Algebra, Geometry and Computer Algebra} at Kaiserslautern
for their kind hospitality.


\section{Prerequisites}\label{sec pre}

Throughout this section, let $F$ be a field of characteristic $p>0$.
We begin by introducing some basic notation that we shall use repeatedly throughout subsequent sections.
Whenever $G$ is a finite group, $FG$-modules are always understood to be finite-dimensional left modules.
Whenever $H$ and $K$ are subgroups of $G$ such that $H$ is $G$-conjugate to a subgroup of $K$, we write
$H\leq_G K$. If $H$ and $K$ are $G$-conjugate then we write $H=_G K$. For $g\in G$, we set ${}^gH:=gHg^{-1}$.

We assume the reader to be familiar with the basic concepts of the representation theory of the symmetric groups.
For background information we refer to \cite{J,JK}.
As usual, for $n\in\NN$, we shall denote the Specht $F\mathfrak{S}_n$-module labelled by a partition $\lambda$ of $n$ by $S^\lambda$.
If $\lambda$ is a $p$-regular partition of $n$ then we shall denote the simple $F\mathfrak{S}_n$-module
$S^\lambda/\Rad(S^\lambda)$ by $D^\lambda$.

\begin{noth}{\bf Brauer constructions and vertices.}\label{noth Brauer}
(a)\, Let $G$ be a finite group, let $M$ be an $FG$-module, and let $P$ be a $p$-subgroup of $G$.
The {\sl Brauer construction} of $M$ with respect to $P$ is defined as
\begin{equation}\label{eqn M(P)}
M(P):=M^P/\sum_{Q<P}\Tr_Q^P(M^Q)\,,
\end{equation}
where $M^P$ denotes the set of $P$-fixed points of $M$, and $\Tr_Q^P:M^Q\to M^P,\; m\mapsto \sum_{xQ\in P/Q} xm$
denotes the relative trace map. The latter is independent of the choice
of representatives of the left cosets $P/Q$. The $FG$-module structure of $M$
induces an $FN_G(P)$-module structure on the $F$-vector space $M(P)$, and $P$ acts trivially on $M(P)$.
Set $\Tr^P(M):=\sum_{Q<P}\Tr_Q^P(M^Q)$.

Moreover, if $R<Q<P$ then $\Tr_R^P=\Tr_Q^P\circ\Tr_R^Q$. Thus
$M(P)=M^P/\sum_{Q<_{\mathrm{max}}P}\Tr_Q^P(M^Q)$, where $Q<_{\mathrm{max}}P$ denotes a maximal subgroup of $P$.
If $Q<_{\mathrm{max}}P$ then every element $g\in P\smallsetminus Q$ has the property that $\{1,g,g^2,\ldots,g^{p-1}\}$ is a set of representatives of the left cosets of $Q$ in $P$; in particular, we get 
$\Tr_Q^P(m)=m+gm+\cdots +g^{p-1}m$, for $m\in M^Q$.

\smallskip

(b)\, Suppose that $M$ is an indecomposable $FG$-module. Then a {\sl vertex} of $M$ is a subgroup
$Q$ of $G$ that is minimal with respect to the property that $M$ is isomorphic to a direct summand of $\ind_Q^G(\res_Q^G(M))$. By \cite{Gr},
the vertices of $M$ form a $G$-conjugacy class of $p$-subgroups of $G$. Moreover, if $R\leq G$ is a $p$-subgroup
such that $M(R)\neq \{0\}$ then $R\leq_G Q$, by \cite[(1.3)]{Br}. The converse is, however, not true in general.
\end{noth}

For proofs of the abovementioned properties of Brauer constructions, see \cite{Br}.
Details on the theory of vertices of indecomposable $FG$-modules can be found in \cite[Section~9]{Alperin} or \cite[Section~4.3]{NT}.
The following will be very useful for proving Theorem~\ref{thm main} in Section~\ref{sec proof} below.
The proof is straightforward, and is thus left to the reader.

\begin{prop}\label{prop Brauer}
Let $G$ be a finite group, let $M$ be an $FG$-module with $F$-basis $B$, and let $P\leq G$ be a $p$-group.
Suppose that there is some $b_0\in B$ satisfying the following properties:

\smallskip

{\rm (i)}\, $b_0\in M^P$;

{\rm (ii)}\, whenever $Q<_{\max} P$, $u\in M^Q$ and $\Tr_Q^P(u)=\sum_{b\in B}a_b(u) b$, for $a_b(u)\in F$, one has
$a_{b_0}(u)=0$.

\noindent
Then $b_0+\Tr^P(M)\in M(P)\smallsetminus\{0\}$.
\end{prop}

Next we shall recall some well-known properties of the simple $F\mathfrak{S}_n$-modules labelled by hook partitions
$(n-r,1^r)$, for $r\in\{0,\ldots,p-1\}$, that we shall need repeatedly in the proof of Theorem~\ref{thm main}.
In particular, we shall fix a convenient $F$-basis of $D^{(n-r,1^r)}$. In light of Theorem~\ref{thm main}
we shall only be interested in the case where $p\mid n$ and $p>2$.

\begin{noth}{\bf Exterior powers of the natural $F\mathfrak{S}_n$-module.}\label{noth exterior}
(a)\, Let $p>2$, let $n\in\NN$ be such that $p\mid n$, and let
$M:=M^{(n-1,1)}$ be the natural permutation $F\mathfrak{S}_n$-module, with natural permutation basis
$\Omega=\{\omega_1,\ldots,\omega_n\}$. Since $p\mid n$, the module $M$ is uniserial with composition series
$\{0\}\subset M_2\subset M_1\subset M$, where $M_1=\{\sum_{i=1}^na_i\omega_i: a_1,\ldots,a_n\in F,\, \sum_{i=1}^na_i=0\}$
and $M_2=\{a\sum_{i=1}^n\omega_i: a\in F\}$; see, for instance, \cite[Example~5.1]{J}.

Furthermore, $M_1=S^{(n-1,1)}$, and $M_1/M_2=:\Hd(S^{(n-1,1)})\cong D^{(n-1,1)}$; in particular, $\dim_F(D^{(n-1,1)})=n-2$.
One sometimes calls $D^{(n-1,1)}$ the {\sl natural (simple) $F\mathfrak{S}_n$-module}.
An $F$-basis of $M_1$ is given by the elements $\omega_i-\omega_1$, where $i\in\{2,\ldots,n\}$.
In the following, we shall identify the module $D^{(n-1,1)}$ with $M_1/M_2$. 

Consider the natural epimorphism ${}^-:M_1\to M_1/M_2$, and set $e_i:=\overline{\omega_i-\omega_1}$, for $i\in\{1,\ldots,n\}$.
Then $e_n=-e_2-e_3-\cdots -e_{n-1}$, and the elements $e_2,\ldots, e_{n-1}$ form an $F$-basis of $D^{(n-1,1)}$.

\smallskip

(b)\, Let $r\in\{0,\ldots,n-1\}$. By \cite[Proposition~2.2]{MZ}, there is an $F\mathfrak{S}_n$-isomorphism $S^{(n-r,1^r)}\cong \bigwedge^rS^{(n-1,1)}$. Moreover, if $r\leq n-2$ then,
in consequence of \cite{P}, 
$\Hd(\bigwedge^rS^{(n-1,1)})\cong \bigwedge^r \Hd(S^{(n-1,1)})\cong \bigwedge^r D^{(n-1,1)}=:D_r$
is simple.
Thus $D_r$ has $F$-basis
\begin{equation}\label{eqn B_r}
\mathcal{B}_r:=\{e_{i_1}\wedge e_{i_2}\wedge\cdots\wedge e_{i_r}: 2\leq i_1<i_2<\cdots < i_r\leq n-1\}\,.
\end{equation}
If $r\leq p-1$ then $\bigwedge^r D^{(n-1,1)}\cong D^{(n-r,1^r)}$. 
\end{noth}


\section{Symmetric Groups and $p$-Subgroups}\label{sec sym}

Throughout this section, let $n\in\NN$, and let $p$ be a prime number. Permutations in the symmetric group
$\mathfrak{S}_n$ will be composed from right to left, so that, for instance, we have $(1,2)(2,3)=(1,2,3)\in\mathfrak{S}_3$.

\begin{defi}\label{defi supp}
Given an element $\sigma\in\mathfrak{S}_n$, we call
$\supp(\sigma):=\{i\in\{1,\ldots,n\}: \sigma(i)\neq i\}$ 
the {\sl support
of $\sigma$}. If $H\leq\mathfrak{S}_n$ then we call $\supp(H):=\bigcup_{\sigma\in H}\supp(\sigma)$ the
{\sl support of $H$}.
\end{defi}

\begin{noth}{\bf Sylow subgroups of symmetric groups.}\label{noth sylow}
(a)\, Let $P_p$ be the cyclic group $\langle (1,2,\ldots,p)\rangle\leq \mathfrak{S}_p$ of order $p$. Let
further $P_1:=\{1\}$ and, for $d\geq 1$, we set 
$$P_{p^{d+1}}:=P_{p^d}\wr P_p:=\{(\sigma_1,\ldots,\sigma_p;\pi): \sigma_1,\ldots,\sigma_p\in P_{p^d},\, \pi\in P_p\}\,.$$
Recall that, for $d\geq 2$, the multiplication in $P_{p^{d}}$ is given by
$(\sigma_1,\ldots,\sigma_p;\pi)(\sigma_1',\ldots,\sigma_p';\pi')=(\sigma_1\sigma_{\pi^{-1}(1)}',\ldots,\sigma_p\sigma_{\pi^{-1}(p)}';\pi\pi')\,,$
for $(\sigma_1,\ldots,\sigma_p;\pi), \, (\sigma_1',\ldots,\sigma_p';\pi')\in P_{p^{d}}$. 

We shall always identify $P_{p^d}$ with a 
subgroup of $\mathfrak{S}_{p^d}$ in the usual way. That is, $(\sigma_1,\ldots,\sigma_p;\pi)\in P_{p^d}$ is identified
with the element $\overline{(\sigma_1,\ldots,\sigma_p;\pi)}\in\mathfrak{S}_{p^d}$ that is defined as follows: if $j\in\{1,\ldots,p^d\}$
is such that $j=p^{d-1}(a-1)+b$, for some $a\in\{1,\ldots,p\}$ and some $b\in\{1,\ldots,p^{d-1}\}$ then
$\overline{(\sigma_1,\ldots,\sigma_p;\pi)}(j):=p^{d-1}(\pi(a)-1)+\sigma_{\pi(a)}(b)$.
Via this identification, $P_{p^d}$ can be generated by the elements $g_1,\ldots, g_d\in\mathfrak{S}_{p^d}$, where
\begin{equation}\label{eqn g_j}
g_j:=\prod_{k=1}^{p^{j-1}}(k,k+p^{j-1},k+2p^{j-1},\ldots, k+(p-1)p^{j-1})\quad (1\leq j\leq d)\,.
\end{equation}
In particular, with this notation we have $P_p\leq P_{p^2}\leq \cdots \leq P_{p^{d-1}}\leq P_{p^d}$, and
 the base group of the wreath product $P_{p^{d-1}}\wr P_p$ has the form 
$\prod_{i=0}^{p-1} g_d^i\cdot P_{p^{d-1}} \cdot g_d^{-i}$.

\medskip

(b)\, Now let $n\in \NN$ be arbitrary, and consider the $p$-adic expansion $n=\sum_{i=0}^r n_ip^i$ of $n$, where
$0\leq n_i\leq p-1$ for $i\in\{0,\ldots, r\}$, and 
where we may suppose that $n_r\neq 0$. 
By \cite[4.1.22, 4.1.24]{JK}, the Sylow $p$-subgroups of $\mathfrak{S}_n$ are
isomorphic to the direct product $\prod_{i=0}^r(P_{p^i})^{n_i}$. For subsequent computations it will be useful to
fix a particular Sylow $p$-subgroup $P_n$ of $\mathfrak{S}_n$ as follows: for $i\in\{t\in\mathbb{N}\ |\ n_t\neq 0\}$ and 
$1\leq j_i\leq n_i$, let $k(j_i):=\sum_{l=0}^{i-1}n_lp^l+(j_i-1)p^i$ and
$$P_{p^i,j_i}:= (1,1+k(j_i))\cdots (p^i,p^i+k(j_i))\cdot P_{p^i}\cdot (1,1+k(j_i))\cdots (p^i,p^i+k(j_i))\,.$$ 
Now set 
$$P_n:=P_{p,1}\times\cdots\times P_{p,n_1}\times\cdots\times P_{p^r,1}\times\cdots\times P_{p^r,n_r}\,.$$
Given this convention, we shall then also write $P_n=\prod_{i=0}^r(P_{p^i})^{n_i}$, for simplicity.
\end{noth}

\begin{expl}\label{expl sylow}
Suppose that $p=3$. Then $P_3=\langle g_1\rangle$, $P_9=\langle g_1,g_2\rangle$
and $P_{27}=\langle g_1,g_2,g_3\rangle$, where
\begin{align*}
g_1&=(1,2,3)\,,\\
g_2&=(1,4,7)(2,5,8)(3,6,9)\,,\\
g_3&=(1,10,19)(2,11,20)(3,12,21)(4,13,22)(5,14,23)(6,15,24)(7,16,25)(8,17,26)(9,18,27)\,.
\end{align*}
Moreover, $P_{51}=P_3\times P_3\times P_9\times P_9\times P_{27}$.
\end{expl}

\begin{noth}{\bf Elementary abelian groups.}\label{noth E_n}
(a)\, Suppose again that $n=p^d$, for some $d\in\NN$. 
We shall denote by $E_n$ the following elementary abelian subgroup of $P_n$ that
acts regularly on $\{1,\ldots,n\}$: let $g_1,\ldots,g_d$ be the
generators of $P_n$ fixed in (\ref{eqn g_j}). For $j\in\{1,\ldots,d-1\}$, let $g_{j,j+1}:=\prod_{i=0}^{p-1} g_{j+1}^i g_j g_{j+1}^{-i}$,
and for $l\in\{1,\ldots,d-j-1\}$, we inductively set 
$$g_{j,j+1,\ldots,j+l+1}:=\prod_{i=0}^{p-1} g_{j+l+1}^i \cdot g_{j,j+1,\ldots,j+l} \cdot g_{j+l+1}^{-i}\,.$$
Then $E_n:=\langle g_{1,\ldots,d}, g_{2,\ldots, d},\ldots, g_{d-1,d}, g_d\rangle$, and $|E_n|=n=p^d$.
 
 \smallskip

(b)\, Let $n\in\NN$ be arbitrary with $p\mid n$, and let $t,m_1,\ldots,m_t\in\NN_0$ be such that $n=\sum_{i=1}^t m_ip^i$.
For $i\in\{s\in\mathbb{N}\ |\ m_s\neq 0\}$ and $1\leq j_i\leq m_i$,
we set $k(j_i):=\sum_{l=0}^{i-1}m_lp^l+(j_i-1)p^i$ and 
$$E_{p^i,j_i}:=(1,1+k(j_i))\cdots (p^i,p^i+k(j_i))\cdot E_{p^i}\cdot (1,1+k(j_i))\cdots (p^i,p^i+k(j_i))\,.$$
Then $E(m_1,\ldots,m_t)\leq \mathfrak{S}_n$ denotes the elementary abelian group 
$$E_{p,1}\times\cdots \times E_{p,m_1}\times \cdots \times E_{p^t,1}\times\cdots \times E_{p^t,m_t}\,.$$

We emphasize that, unlike in \ref{noth sylow}, the integers $m_1,\ldots,m_t$ need not be less than $p$.
\end{noth}

\begin{expl}\label{expl E_n}
Suppose that $p=3$ and $n=27$. Then $E_n=E_{27}$ is generated by the elements
\begin{align*}
g_{1,2,3}&=(1,2,3)(4,5,6)(7,8,9)(10,11,12)(13,14,15)(16,17,18)(19,20,21)(22,23,24)(25,26,27)\,,\\
g_{2,3}&=(1,4,7)(2,5,8)(3,6,9)(10,13,16)(11,14,17)(12,15,18)(19,22,25)(20,23,26)(21,24,27)\,,\\
g_3&=(1,10,19)(2,11,20)(3,12,21)(4,13,22)(5,14,23)(6,15,24)(7,16,25)(8,17,26)(9,18,27)\,.
\end{align*}
\end{expl}

We recall the following lemma from \cite{D}, which will be
useful for our subsequent considerations.

\begin{lemma}[\protect{\cite[Lemma~2.1, Remark~2.2]{D}}]\label{lemma D 2.1}
Let $n\in\NN$ with $p$-adic expansion $n=\sum_{i=0}^r n_i p^i$, as in \ref{noth sylow}. 
Let $P\leq P_n$ be such
that $P=_{\mathfrak{S}_n} P_{p^i}$, for some $i\in\{1,\ldots,r\}$. Then $P\leq P_{p^l,j_l}$, for some $l\in\{i,\ldots,r\}$ and
some $1\leq j_l\leq n_l$. Moreover,
$P_{p^l,j_l}$ has precisely $p^{l-i}$ subgroups that are $\mathfrak{S}_n$-conjugate to $P_{p^i}$, and these 
are pairwise $P_{p^l,j_l}$-conjugate to each other.
\end{lemma}

\begin{rem}\label{rem E_n P_n}
Let again $n\in\NN$ with $p$-adic expansion $n=\sum_{i=0}^r n_i p^i$.

\smallskip

(a)\, Let $P\leq P_n$ be such
that $P=_{\mathfrak{S}_n} P_{p^i}$, for some $i\in\{1,\ldots,r\}$, so that
$P\leq P_{p^l,j_l}$, for some $l\in\{i,\ldots,r\}$ and
some $1\leq j_l\leq n_l$, by Lemma~\ref{lemma D 2.1}.
Note that the  subgroups of $P_{p^l,j_l}$ that are $\mathfrak{S}_n$-conjugate to $P_{p^i}$ are uniquely determined by their
supports. In particular, if $i=1$ then $P$ is generated by one of the $p$-cycles
$(1,\ldots,p),\ldots, (n-n_0-p+1,\ldots,n-n_0)\in P_n$.

\smallskip

(b)\, Suppose that $E\leq P_n$ is such that $E=_{\mathfrak{S}_n} E_{p^i}$, for some $i\in\{1,\ldots,r\}$. 
Since $E$ has precisely one non-trivial orbit, we then also get $E\leq  P_{p^l,j_l}$, for some $l\in\{i,\ldots,r\}$ and
some $1\leq j_l\leq n_l$. Moreover, arguing by induction on $l-i$ as in the proof of \cite[Lemma 2.1]{D}, we 
deduce that $E$ then has to be contained in one of the $p^{l-i}$ subgroups of $P_{p^l,j_l}$
that are $\mathfrak{S}_n$-conjugate to $P_{p^i}$.
\end{rem}

\begin{lemma}\label{lemma E_n}
Let $n,d\in\NN$, and let $P\leq P_{p^d}\leq \mathfrak{S}_n$. Suppose that
$P$ contains an $\mathfrak{S}_n$-conjugate of $P_{p^{d-1}}$.
Suppose further that $P$ contains an elementary abelian group $E$ of order $p^d$ acting regularly on $\{1,\ldots,p^d\}$. Then
$P=P_{p^d}$.
\end{lemma}

\begin{proof}
If $d=1$ then $P_{p^d}=P_p=E$. From now on we may suppose that $d\geq 2$.
Recall that $P_{p^d}$ is generated by the elements $g_1,\ldots,g_d$ introduced in (\ref{eqn g_j}).
Moreover, $P_{p^d}$ acts imprimitively on the set $\{1,\ldots,p^d\}$, a system of imprimitivity being given
by $\Delta:=\{\Delta_1,\ldots,\Delta_p\}$, where $\Delta_s:=\{(s-1)p^{d-1}+1,\ldots,sp^{d-1}\}$, for $s\in\{1,\ldots,p\}$.
Since $E$ acts transitively on $\{1,\ldots,p^d\}$, there is some $g\in E$ such that $g(1)=p^{d-1}+1$; in particular, $g\cdot \Delta_1=\Delta_2$. Since
$p^{d-1}+1\neq 1$, we have $g\neq 1$, hence $g$ is an element of order $p$.  
Moreover, the group $\langle g\rangle$ acts on $\Delta$, so that we obtain a group homomorphism
$\varphi:\langle g\rangle\to \mathfrak{S}(\Delta)\cong \mathfrak{S}_p$. Since $g\cdot \Delta_1=\Delta_2\neq \Delta_1$, $\varphi$
must be injective. Thus $\varphi(g)$ has order $p$, implying
$g\cdot \Delta_1=\Delta_2$, $g\cdot \Delta_2=\Delta_{i_3},\ldots, g\cdot \Delta_{i_p}=\Delta_1$, for
$\{1,2,i_3,\ldots,i_p\}=\{1,\ldots,p\}$.

Let $R:={}^{\sigma}P_{p^{d-1}}\leq P$, for some $\sigma\in \mathfrak{S}_n$.
By Lemma~\ref{lemma D 2.1}, we know that $R=g_d^i P_{p^{d-1}} g_d^{-i}$, for some $i\in\{0,\ldots,p-1\}$. Thus
$\supp(R)=\Delta_{i+1}$. So, for $s\in\{0,\ldots, p-1\}$, the group ${}^{g^s}R$
has support $g^s\cdot \Delta_{i+1}$. As we have just seen, the sets $\Delta_{i+1}, g\cdot \Delta_{i+1},\ldots, g^{p-1}\cdot\Delta_{i+1}$
are pairwise disjoint.
Consequently, the groups $R$, ${}^gR,\ldots,{}^{g^{p-1}}R$ are precisely the different subgroups
of $P_{p^d}$ that are $P_{p^d}$-conjugate to $P_{p^{d-1}}$, $B:=\prod_{s=0}^{p-1}{}^{g^s}R$ is the base group 
of $P_{p^d}$, and is contained in $P$. Clearly $g\notin B$, since $g(1)\notin \Delta_1$.
Since $[P_{p^d}:B]=p$, this implies $P_{p^d}=\langle B, g\rangle\leq P\leq P_{p^d}$, and the proof is complete.
\end{proof}

\begin{lemma}\label{lemma max in E}
Let $n,t\in\NN$ and let $m_1,\ldots,m_t\in\NN_0$ be such that $m_t\neq 0$ and $n=\sum_{i=1}^tm_i p^i$.  
Suppose that $m_1=1$ and
$t\geq 2$. Let $P$ be a maximal subgroup of $E(m_1,\ldots,m_t)$ such that $E_{p,1}\not\leq P$. Then
$P$ contains a subgroup $Q\leq \prod_{i=2}^t \prod_{j=1}^{m_i} E_{p^i,j}$ that acts fixed point freely
on $\{p+1,\ldots,n\}$.
\end{lemma}

\begin{proof}
For convenience, set $E':=\prod_{i=2}^t \prod_{j=1}^{m_i} E_{p^i,j}$, so that $E(m_1,\ldots,m_t)=E_p\times E'\geq P$.
By Goursat's Lemma, we may identify $P$ with the quintuple $(P_1,K_1,\eta,P_2,K_2)$, where
$P_1$ and $P_2$ are the projections of $P$ onto $E_p$ and onto $E'$, respectively, 
$K_1:=\{g\in E_p: (g,1)\in P\}\unlhd P_1$, $K_2:=\{h\in E': (1,h)\in P\}$, and $\eta: P_2/K_2\to P_1/K_1$ is a 
group isomorphism. Since $|E_p|=p$, there are precisely three possibilities for the section $(P_1,K_1)$ of $E_p$:

\smallskip

(i)\, $P_1=K_1=E_p$,

(ii)\, $P_1=K_1=\{1\}$,

(iii)\, $P_1=E_p$ and $K_1=\{1\}$.

\noindent
Case (i) cannot occur, since we are assuming $E_p\not\leq P$. In case (ii) we get $P=E'$, so that the assertion then holds with
$Q:=P$. So suppose that $P_1=E_p$ and $K_1=\{1\}$, so that also $[P_2:K_2]=p$.
Next recall that $P/(K_1\times K_2)\cong P_1/K_1\cong P_2/K_2$; see, for instance, \cite[2.3.21]{Bouc}.
This forces
$|E'|=|P|=|K_2|\cdot |P_1|=|K_1|\cdot |P_2|=|P_2|$. Thus $P_2=E'$, and $K_2$ is a maximal subgroup of $E'$.
Assume that $K_2$ has a fixed point $x$ on $\{p+1,\ldots,n\}$. Then $x\in\supp(E_{p^i,j})$,
 for some $i\geq 2$ with $m_i\neq 0$ and  some $j\in\{1,\ldots,m_i\}$.
But then $K_2$ has to fix the entire support 
of $E_{p^i,j}$, since $E_{p^i,j}$ acts regularly on its support.
This implies $[P_2:K_2]\geq p^i\geq p^2$, a contradiction.
Consequently, $K_2$ must act fixed point freely on $\{p+1,\ldots,n\}$, and the assertion of the lemma follows 
with $Q:=\{1\}\times K_2\leq P$.
\end{proof}

The next result will be one of the key ingredients of our proof of Theorem~\ref{thm main} in Section~\ref{sec proof} below.

\begin{prop}\label{prop P and E in Q}
Let $n\in\NN$ with $p$-adic expansion $n=p+\sum_{i=2}^rn_ip^i$, where $r\geq 2$ and $n_r\neq 0$. 
Let $Q\leq P_n$ be such that $P_{n-2p}\leq_{\mathfrak{S}_n} Q$ and $E(1,n_2,\ldots,n_r)\leq_{\mathfrak{S}_n} Q$.
Then $Q=P_n$.
\end{prop}

\begin{proof}
Let $2\leq s\leq r$ be minimal such that $n_s\neq 0$. Then
$n-2p$ has $p$-adic expansion $n-2p=\sum_{j=1}^{s-1} (p-1)p^j+(n_s-1)p^s+\sum_{i=s+1}^rn_ip^i$. Moreover, we have
$$P_n=P_{p,1}\times\prod_{i=s}^r\prod_{j=1}^{n_i} P_{p^i,j}\quad \text{ and } \quad E_n=E(1,n_2,\ldots,n_r)=E_{p,1}\times\prod_{i=s}^r\prod_{j=1}^{n_i} E_{p^i,j}\,.$$
By our hypothesis, there is some $g\in\mathfrak{S}_n$ such that
${}^gE_{p,1}\times\prod_{i=s}^r\prod_{j=1}^{n_i} {}^gE_{p^i,j}\leq Q\leq P_n$.
In consequence of Lemma~\ref{lemma D 2.1} and Remark~\ref{rem E_n P_n}, 
we may suppose that ${}^gE_{p^i,j}\leq P_{p^i,j}$, for $i\geq 2$ and $1\leq j\leq n_i$, as well as ${}^gE_{p,1}=E_{p,1}=P_{p,1}$.
Since also $P_{n-2p}\leq_{\mathfrak{S}_n} Q$, there exists some $R\leq Q\leq P_n$ of the form
$$R=\prod_{i=1}^{s-1}\prod_{j=1}^{p-1}R_{p^i,j}\times \prod_{j=1}^{n_s-1}R_{p^s,j}\times\prod_{i=s+1}^r \prod_{j=1}^{n_i}R_{p^i,j}\,,$$
where $R_{p^k,l}=_{\mathfrak{S}_n} P_{p^k,l}$, for all possible $k$ and $l$. By Lemma~\ref{lemma D 2.1}
and Remark~\ref{rem E_n P_n} again, we must have 
$\prod_{i=s+1}^r \prod_{j=1}^{n_i}R_{p^i,j}=\prod_{i=s+1}^r \prod_{j=1}^{n_i}P_{p^i,j}\leq P_n$. As well, there is 
some $k\in\{1,\ldots,n_s\}$ and some $m\in\{1,\ldots,p-1\}$
such that
$ \prod_{j=1}^{n_s-1}R_{p^s,j}=\prod_{j=1}^{k-1}P_{p^s,j}\times \prod_{l=k+1}^{n_s} P_{p^s,l}\leq P_n$
and $R_{p^{s-1},m}\leq P_{p^s,k}$. 
By Lemma~\ref{lemma D 2.1}, $R_{p^{s-1},m}$ is thus $P_{p^s,k}$-conjugate
to one of the $p^{s-1}$ subgroups of $P_{p^s,k}$ that are $\mathfrak{S}_n$-conjugate to $P_{p^{s-1}}$.
Since $Q$ also contains the regular elementary abelian group ${}^gE_{p^s,k}\leq P_{p^s,k}$, Lemma~\ref{lemma E_n}
now implies that $P_{p^s,k}\leq Q$. Altogether this shows that indeed $P_n\leq Q$, and the assertion of the proposition follows.
\end{proof}


\section{The Proof of Theorem~\ref{thm main}}\label{sec proof}

The aim of this section is to establish a proof of Theorem~\ref{thm main}. 
To this end, let $F$ be a field of characteristic $p>2$, and let $n\in \NN$ be such that 
$n\equiv p\pmod{p^2}$.  
The simple $F\mathfrak{S}_n$-module $D^{(n-p+1,1^{p-1})}$
will henceforth be denoted by $D$. 
If $p=n$ then the Sylow $p$-subgroups of $\mathfrak{S}_n$ are abelian, and are thus the vertices of $D$, by Kn\"orr's Theorem \cite{Kn}.
From now on we shall suppose that $n\geq p^2+p$.
Let $P_n$ be the Sylow $p$-subgroup of $\mathfrak{S}_n$ introduced
in \ref{noth sylow}. In order to show that $P_n$ is a vertex of $D$, we shall proceed as follows: suppose that $Q\leq P_n$ is a 
vertex of $D$. Then:

\smallskip

(i)\, Building on Wildon's result in \cite[Theorem~2]{W}, it was shown in
\cite[Proposition~5.2]{D} that $P_{n-2p}=P_{n-(p-1)-2}\times P_{p-1}<_{\mathfrak{S}_n} Q$.

\smallskip

(ii)\, Let $n=\sum_{i=2}^rn_i p^i+p$ be the $p$-adic expansion of $n$, where $r\geq 2$ and $n_r\neq 0$. We shall show
in Proposition~\ref{prop Brauer D(E)} below that $D(E(1,n_2,\ldots,n_r))\neq \{0\}$. 
Here $E(1,n_2,\ldots,n_r)$ denotes the elementary abelian subgroup of $P_n$ defined in \ref{noth  E_n},
and $D(E(1,n_2,\ldots,n_r))$ denotes the Brauer construction of $D$ with respect to $E(1,n_2,\ldots,n_r)$
as defined in \ref{noth Brauer}.
Thus, $E(1,n_2,\ldots,n_r)\leq_{\mathfrak{S}_n} Q$, by \cite[(1.3)]{Br}.

\smallskip

(iii)\, Once we have verified (ii), we can apply Proposition~\ref{prop P and E in Q},  which will then show that $Q=P_n$.

\begin{nota}\label{nota e}
(a)\, Let $\mathcal{B}:=\mathcal{B}_{p-1}$ be the $F$-basis of $D$ defined in (\ref{eqn B_r}), and let $u\in D$ be such that
$u=\sum_{b\in\mathcal{B}}\lambda_b  b$, for $\lambda_b\in F$. The basis element $e_2\wedge e_3\w \cdots \w e_p\in\mathcal{B}$ will from now on be denoted by $e$.
Moreover, suppose that $k,x\in\{2,\ldots,n-1\}$ and that $k\leq p$. Then we denote the element 
$e_2\wedge \cdots \wedge e_{k-1}\wedge e_{k+1}\wedge\cdots \wedge e_p\wedge e_x$ of $D$
by $\he_k\w e_x$. In the case where $\he_k\w e_x\in\mathcal{B}$, the
coefficient $\lambda_{e_2\wedge \cdots \wedge e_{k-1}\wedge e_{k+1}\wedge\cdots \wedge e_p\wedge e_x}$ will be abbreviated by 
$\lambda_{\hk,x}$.

Similarly, if $2\leq k<l\leq p$ and if $x,y\in\{2,\ldots,n-1\}$, then we set 
$\he_{k,l}\w e_x\w e_y:=e_2\w\cdots\w e_{k-1}\w e_{k+1}\w\cdots \w e_{l-1}\w e_{l+1}\w\cdots \w e_p\w e_x\w e_y\in D$.
In the case where $\he_{k,l}\w e_x\w e_y\in\mathcal{B}$, we
denote by $\lambda_{\widehat{k,l},x,y}$ the coefficient
at $\he_{k,l}\w e_x\w e_y$ in $u$.

\smallskip

(b)\, Let $u\in D$ be such that $u=\sum_{b\in \mathcal{B}} \lambda_b b$, with $\lambda_b\in F$.
We say that the basis element $b\in\mathcal{B}$ {\sl occurs in $u$} with coefficient $\lambda_b$.

\smallskip

(c)\, For $k_1,k_2\in\{2,\ldots,n-1\}$, we set
\begin{equation}\label{eqn s(k)}
s(k_1,k_2):= \begin{cases}
k_2-(k_1-1) &\text{ if } k_1\leq k_2\,,\\
0&\text{ if } k_2< k_1\,.
\end{cases}
\end{equation}
Thus, if $k_1\leq k_2$ then 
$$s(k_1,k_2)\equiv \begin{cases}
0\pmod{ 2}&\text{ if } k_1\not\equiv k_2\pmod{2}\,,\\
1\pmod{2}&\text{ if } k_1\equiv k_2\pmod{2}\,.
\end{cases}$$

\smallskip

(d)\, From now on, let 
$t,m_2,\ldots,m_t\in \NN$ be such that $t\geq 2$, $m_t\neq 0$, and
$n=p+\sum_{i=2}^t m_i p^{i}$.
The elementary abelian group $E(1,m_2,\ldots,m_t)\leq \mathfrak{S}_n$ will be denoted by $E$. Note that, by our convention
in \ref{noth E_n}, we have $(1,2,\ldots,p)\in E$. In the case where $t=r$ and $m_i=n_i$, for $i=2,\ldots,r$, we, in particular, 
get $E=E(1,n_2,\ldots,n_r)$.
\end{nota}

In the course of this section we shall have to compute explicitly the actions of elements in $E$ on
our chosen basis $\mathcal{B}$ of $D$. 
The following lemmas will be used repeatedly in this section.

\begin{lemma}\label{lemma acts alpha}
Let $\alpha:=(1,2,\ldots,p)\in \mathfrak{S}_n$.
Let further $\beta:=(x_1,\ldots,x_p)\in\mathfrak{S}_n$ be such that $\{x_1,\ldots,x_p\}\cap \{1,\ldots,p\}=\emptyset$.

\smallskip

{\rm (a)}\, For $i\in\{2,\ldots,n-1\}$, one has
$$\alpha\cdot e_i=\begin{cases}
e_{i+1}-e_2&\text{ if } 2\leq i\leq p-1\,,\\
-e_2&\text{ if } i=p\,,\\
e_i-e_2&\text{ if } i\geq p+1\,.
\end{cases}$$

\smallskip

{\rm (b)}\, If $n\notin\supp(\beta)$ then, for $i\in\{2,\ldots,n-1\}$, one has
$$\beta\cdot e_i=\begin{cases}
e_i&\text{ if } i\notin\supp(\beta)\,,\\
e_{\beta(i)}&\text{ if } i\in\supp(\beta)\,.
\end{cases}$$

\smallskip

{\rm (c)}\, If $x_p=n$ then, for  $i\in\{2,\ldots,n-1\}$, one has
$$\beta\cdot e_i=\begin{cases}
e_i&\text{ if } i\notin\supp(\beta)\,,\\
e_{\beta(i)}&\text{ if } i\in\{x_1,\ldots,x_{p-2}\}\,,\\
-\sum_{j=2}^{n-1}e_j&\text{ if } i=x_{p-1}\,.
\end{cases}$$
\end{lemma}

\begin{proof}
{\rm (a)}\, If $2\leq i\leq p-1$, then 
\begin{align*}
\alpha \cdot e_i=&\alpha\cdot (\overline{\omega_i-\omega_1})=\overline{\alpha\cdot (\omega_i-\omega_1)}
=\overline{\omega_{\alpha(i)}-\omega_{\alpha(1)}}=\overline{\omega_{i+1}-\omega_2}
=\overline{(\omega_{i+1}-\omega_1)-(\omega_2-\omega_1)}\\
=& e_{i+1}-e_2.
\end{align*}
If $i=p$, then $\alpha \cdot e_i=\alpha\cdot (\overline{\omega_p-\omega_1})=\overline{\omega_1-\omega_2}=-e_2$.
Finally, if $i\geq p+1$, then we have 
$$\alpha \cdot e_i=\overline{\omega_i-\omega_2}=\overline{(\omega_i-\omega_1)-(\omega_2-\omega_1)}=e_i-e_2\,.$$

\smallskip

The proofs of {\rm (b)} and {\rm (c)} are similar, and are left to the reader. 
\end{proof}

\begin{lemma}\label{lemma s(k)}
Let $k,l\in\{2,\ldots,p\}$, and let $x\in\{p+1,\ldots,n-1\}$. Then one has

\smallskip


{\rm (a)}\, $e_{k+1}\w\cdots \w e_p\w e_2\w\cdots \w e_{k-1}\w e_x=(-1)^{s(k+1,p)(k-2)} \he_k\w e_x$;

\smallskip

{\rm (b)}\, $\he_k\w e_k=(-1)^{s(k+1,p)}  e$;

\smallskip

{\rm (c)}\, if $k<l$ then $\he_{k,l}\w e_x\w e_l= (-1)^{s(l+1,p)+1}  \he_k\w e_x$;

\smallskip

{\rm (d)}\, if $k<l$ then $\he_{k,l}\w e_x\w e_k=(-1)^{s(k+1,p)} \he_l\w e_x$.
\end{lemma}

\begin{proof}
{\rm (a)}\, For $k\in\{2,\ldots,p\}$ and $x\in\{p+1,\ldots,n-1\}$, we have
\begin{align*}
&\overbrace{e_{k+1}\w\cdots \w e_p}^{s(k+1,p)}\w \underbrace{e_2\w\cdots \w e_{k-1}}_{k-2}\w e_x\\
&=(-1)^{s(k+1,p)} e_2\w e_{k+1}\w\cdots \w e_p\w e_3\w\cdots \w e_{k-1}\w e_x=(-1)^{s(k+1,p)(k-2)} \he_k\w e_x.
\end{align*}

The proofs of {\rm (b)}, {\rm (c)} and {\rm (d)} are similar, and are left to the reader. 
\end{proof}

\begin{cor}\label{cor e}
For $e:=e_2\wedge e_3\w \cdots \w e_p$, we have $e\in D^{P_n}$; in particular, $e\in D^P$, for every $P\leq P_n$.
\end{cor}

\begin{proof}
With the notation in \ref{noth sylow} we have $P_n=P_p\times\prod_{i=2}^r(P_{p^i})^{n_i}$, and $P_p=\langle \alpha\rangle$,
where $\alpha:=(1,2,\ldots,p)$.
If $\beta\in \prod_{i=2}^r(P_{p^i})^{n_i}$ then we clearly have $\beta\cdot e=e$.
By Lemma~\ref{lemma acts alpha} and Lemma~\ref{lemma s(k)}(b), we also have
$$\alpha\cdot e=(e_3-e_2)\w (e_4-e_2)\w \cdots \w (e_p-e_2)\w (-e_2)=(-1)^{s(3,p)+1} e=(-1)^2 e=e\,.$$
\end{proof}

\begin{lemma}\label{lemma sigma}
Let $1\neq \sigma\in E$, and let $q\in\NN$ be such that 
$$\sigma=(x_1^1,\ldots,x_p^1)\cdots (x_1^q,\ldots,x_p^q)\,,$$
where 
$\{x_i^s: 1\leq i\leq p\,, 1\leq s\leq q\}=\supp(\sigma)\subseteq \{p+1,\ldots,n\}$ and $x_p^q=n$. Let further $u\in D$ be such
that $u=\sum_{b\in\mathcal{B}} \lambda_b b$, for $\lambda_b\in F$. Suppose that $\sigma\cdot u=u$.
Then one has the following:

\smallskip

{\rm (a)}\, $\sum_{k=2}^p(-1)^k \lambda_{\hk,x_i^q}=0$, for every $i\in\{1,\ldots,p-1\}$;

\smallskip

{\rm (b)}\, $\sum_{k=2}^p(-1)^{k+1}  \lambda_{\hk,x_i^s}=\sum_{k=2}^p(-1)^{k+1}\lambda_{\hk,x_1^s}$, for
$i\in\{1,\ldots,p\}$ and $1\leq s\leq q-1$.
\end{lemma}

\begin{proof}
Let
$x\in\{x_i^s: 1\leq i\leq p,\, 1\leq s\leq q-1\}$, and let $k\in\{2,\ldots,p\}$. Suppose that $b\in\mathcal{B}$ is such that
$\he_k\w e_x$ occurs with non-zero coefficient in $\sigma\cdot b$. Then

\smallskip

(i)\, $b=\he_k\w e_{\sigma^{-1}(x)}$, or

(ii)\, $b=\he_k\w e_{x_{p-1}^q}$, or

(iii)\, $b=\he_{k,k_2}\w e_{\sigma^{-1}(x)}\w e_{x_{p-1}^q}$ and $\sigma^{-1}(x)<x_{p-1}^q$, for some $k<k_2\leq p$, or

(iv)\,  $b=\he_{k_1,k}\w e_{\sigma^{-1}(x)}\w e_{x_{p-1}^q}$ and $\sigma^{-1}(x)<x_{p-1}^q$, for some $2\leq k_1<k$, or

(v)\, $b=\he_{k,k_2}\w  e_{x_{p-1}^q}\w e_{\sigma^{-1}(x)}$ and $\sigma^{-1}(x)>x_{p-1}^q$, for some $k<k_2\leq p$, or

(vi)\,  $b=\he_{k_1,k}\w  e_{x_{p-1}^q}\w e_{\sigma^{-1}(x)}$ and $\sigma^{-1}(x)>x_{p-1}^q$, for some $2\leq k_1<k$.

\smallskip

\noindent
If $b$ is one of the basis elements in (i)--(vi) then
the following table records $\sigma \cdot b$ as well as the coefficient at 
$\he_k\w e_x$  in $\sigma \cdot b$, which is obtained using Lemma~\ref{lemma s(k)}.

\medskip

\begin{tabular}{|l|l|c|}\hline
$b$& $\sigma\cdot b$ & coefficient\\\hline\hline
$\he_k\w e_{\sigma^{-1}(x)}$&$\he_k\w e_x$& $1$\\\hline
$\he_k\w e_{x_{p-1}^q}$&$\he_k\w \sum_{y=2}^{n-1}(-e_y)$& $-1$\\\hline
$\he_{k,k_2}\w e_{\sigma^{-1}(x)}\w e_{x_{p-1}^q}$&$\he_{k,k_2}\w e_x\w \sum_{y=2}^{n-1}(-e_y)$&$(-1)^{1+s(k_2+1,p)+1}$\\\hline
$\he_{k_1,k}\w e_{\sigma^{-1}(x)}\w e_{x_{p-1}^q}$&$\he_{k_1,k}\w e_x\w \sum_{y=2}^{n-1}(-e_y)$&$(-1)^{1+s(k_1+1,p)}$\\\hline
$\he_{k,k_2}\w e_{x_{p-1}^q}\w e_{\sigma^{-1}(x)}$&$\he_{k,k_2}\w \sum_{y=2}^{n-1}(-e_y)\w e_x$&$(-1)^{1+s(k_2+1,p)}$\\\hline
$\he_{k_1,k}\w e_{x_{p-1}^q}\w e_{\sigma^{-1}(x)}$&$\he_{k_1,k}\w \sum_{y=2}^{n-1}(-e_y)\w e_x$&$(-1)^{s(k_1+1,p)}$\\\hline

\end{tabular}

\bigskip

Now note that $(-1)^{1+s(k_2+1,p)+1}=(-1)^{1+p-k_2+1}=(-1)^{k_2+1}$ and $(-1)^{1+s(k_1+1,p)}=(-1)^{1+p-k_1}=(-1)^{k_1}.$
Since $\sigma\cdot u=u$, this shows that
\begin{equation}\label{eqn sum 5}
\lambda_{\hk,x}=\lambda_{\hk,\sigma^{-1}(x)}-\lambda_{\hk,x_{p-1}^q}+\sum_{k_2=k+1}^p(-1)^{k_2+1}\lambda_{\widehat{k,k_2},\sigma^{-1}(x),x_{p-1}^q}+\sum_{k_1=2}^{k-1}(-1)^{k_1}\lambda_{\widehat{k_1,k},\sigma^{-1}(x),x_{p-1}^q}\,
\end{equation}
if $\sigma^{-1}(x)<x_{p-1}^q$ and
\begin{equation}\label{eqn sum 5'}
\lambda_{\hk,x}=\lambda_{\hk,\sigma^{-1}(x)}-\lambda_{\hk,x_{p-1}^q}-\sum_{k_2=k+1}^p(-1)^{k_2+1}\lambda_{\widehat{k,k_2},x_{p-1}^q,\sigma^{-1}(x)}-\sum_{k_1=2}^{k-1}(-1)^{k_1}\lambda_{\widehat{k_1,k},x_{p-1}^q,\sigma^{-1}(x)}\,
\end{equation}
if $\sigma^{-1}(x)>x_{p-1}^q$.
Moreover,
\begin{align*}
&\sum_{k=2}^p(-1)^{k+1}\left(\sum_{k_2=k+1}^p(-1)^{k_2+1}\lambda_{\widehat{k,k_2},\sigma^{-1}(x),x_{p-1}^q}+\sum_{k_1=2}^{k-1}(-1)^{k_1}\lambda_{\widehat{k_1,k},\sigma^{-1}(x),x_{p-1}^q}\right)\\
&=\sum_{k=2}^p\sum_{l=k+1}^p ((-1)^{k+1}(-1)^{l+1}+(-1)^k(-1)^{l+1}) \lambda_{\widehat{k,l},\sigma^{-1}(x),x_{p-1}^q}=0\,
\end{align*}
if $\sigma^{-1}(x)<x_{p-1}^q$, and 
\begin{align*}
&\sum_{k=2}^p(-1)^{k+1}\left(-\sum_{k_2=k+1}^p(-1)^{k_2+1}\lambda_{\widehat{k,k_2},x_{p-1}^q,\sigma^{-1}(x)}-\sum_{k_1=2}^{k-1}(-1)^{k_1}\lambda_{\widehat{k_1,k},x_{p-1}^q,\sigma^{-1}(x)}\right)\\
&=-\sum_{k=2}^p\sum_{l=k+1}^p ((-1)^{k+1}(-1)^{l+1}+(-1)^k(-1)^{l+1}) \lambda_{\widehat{k,l},x_{p-1}^q,\sigma^{-1}(x)}=0\,
\end{align*}
if $\sigma^{-1}(x)>x_{p-1}^q$.
Hence, from (\ref{eqn sum 5}) and (\ref{eqn sum 5'}) we get
\begin{equation}\label{eqn sum 2}
\sum_{k=2}^p(-1)^{k+1}\lambda_{\hk,x_i^s}=\sum_{k=2}^p(-1)^{k+1}\lambda_{\hk,\sigma^{-1}(x_i^s)}+\sum_{k=2}^p(-1)^k\lambda_{\hk,x_{p-1}^q}\,,
\end{equation}
for every $i\in\{1,\ldots,p\}$ and $1\leq s\leq q-1$.

\smallskip

We also have $\sigma^i \cdot u=u$, for $i=1,\ldots,p-1$. 
To compare the coefficient at $e$ in $u$ and in $\sigma^i\cdot u$, let $i\in\{1,\ldots,p-1\}$ and suppose 
that $b\in\mathcal{B}$ is such that $e$ occurs in $\sigma^i\cdot b$ with non-zero coefficient. 
Then either $b=e$ and $e=\sigma^i\cdot  e$, or
$b=\he_k\w e_{\sigma^{-i}(x_{p}^q)}$, for some $k\in\{2,\ldots,p\}$. Moreover, in the latter case we have
$\sigma^i\cdot b=\he_k\w (-e_2-e_3-\cdots -e_{n-1})$, where $e$ occurs with coefficient
$$(-1)^{s(k+1,p)+1}=\begin{cases}
1&\text{ if } 2\mid k\,,\\
-1&\text{ if } 2\nmid k\,,
\end{cases}$$
by Lemma~\ref{lemma s(k)}. So we obtain 
$\lambda_e=\lambda_e+\sum_{k=2}^p(-1)^k \lambda_{\hk,\sigma^{-i}(x_p^q)}$,
for $i\in\{1,\ldots,p-1\}$, that is, 
\begin{equation}\label{eqn sum 3}
0=\sum_{k=2}^p(-1)^k \lambda_{\hk,x_j^q}\,,
\end{equation}
for $j\in\{1,\ldots,p-1\}$, which proves assertion~(a). Now assertion~(b) follows from (\ref{eqn sum 2}) and (\ref{eqn sum 3}) with
$j=p-1$.
\end{proof}


Next we shall show that $D(E)\neq \{0\}$, where $E$ is the elementary abelian group
in \ref{nota e}. In order to do so, we want to apply Proposition~\ref{prop Brauer} with $b_0=e$.

\begin{lemma}\label{lemma TrP}
Let $P$ be a maximal subgroup of $E$. If $u\in D^P$ then
$e$ occurs in $\Tr_P^E(u)$ with coefficient $0$.
\end{lemma}

\begin{proof}
Set $\alpha:=(1,2,\ldots,p)$. Let $u\in D^P$, and write $u=\sum_{b\in \mathcal{B}} \lambda_b  b$, where $\lambda_b\in F$.
We shall treat the case where $\alpha\in P$ and the case where $\alpha\notin P$ separately.

\smallskip

\underline{Case 1:} $\alpha\in P$. Then there is some $1\neq g\in \prod_{i=2}^t\prod_{j=1}^{m_i} E_{p^i,j}$
with $g\notin P$. Thus $\{1,g,g^2,\ldots,g^{p-1}\}$ is a set of representatives of the left cosets of $P$ in $E$, so that we get 
$\Tr_P^E(u)=u+g\cdot u+\cdots +g^{p-1}\cdot u=\sum_{b\in\mathcal{B}}\sum_{i=0}^{p-1} \lambda_b  (g^i \cdot b)$.

\smallskip

Since $g\neq 1$ and $t\geq 2$, we have
$$g=(x_1^1,\ldots,x_p^1)\cdots (x_1^q,\ldots,x_p^q)\,,$$
for some $q\geq p$ and $\{x_i^s: 2\leq i\leq p,\, 1\leq s\leq q\}=\supp(g)$. 

Suppose first that $n\notin\supp(g)$, and let $b\in\mathcal{B}$. Let further $i\in\{0,\ldots,p-1\}$, and suppose
that $e$ occurs in $g^i \cdot b$ with non-zero coefficient. Then we must have $b=e$, in which case
$\sum_{i=0}^{p-1} g^i \cdot b=pe=0$, by Corollary~\ref{cor e}; in particular, $e$ occurs in $\Tr_P^E(u)$ with coefficient 0.

\smallskip

So we may now suppose that $n\in\supp(g)$. Moreover, we may suppose that $x_p^q=n$. Let $i\in\{0,\ldots,p-1\}$, and let $b\in\mathcal{B}$ be such that $e$ occurs in $g^i \cdot b$ with non-zero coefficient.
If $i=0$ then we must of course have $b=e=g^0\cdot e$. If $i\geq 1$ then $b=e$, or $b=\he_k\w e_{g^{-i}(x_p^q)}$, for some
$k\in\{2,\ldots,p\}$. In the latter case, we have 
$g^i\cdot (\he_k\w e_{g^{-i}(x_p^q)})=\he_k\w (-e_2-e_3-\cdots -e_{n-1})$, in which $e$ occurs with coefficient
$$(-1)^{s(k+1,p)+1} =\begin{cases}
1&\text{ if } 2\mid k\,,\\
-1&\text{ if } 2\nmid k\,,
\end{cases}$$
by Lemma~\ref{lemma s(k)}.
Consequently, the coefficient at $e$ in $\Tr_P^E(u)$ equals
\begin{equation}\label{eqn e in Tru}
p\lambda_e+\sum_{i=1}^{p-1}\left(\mathop{\sum_{k=2}^p}_{2\mid k}\lambda_{\hk,x_i^q}-\mathop{\sum_{l=2}^p}_{2\nmid l}\lambda_{\hl,x_i^q}\right)=\sum_{i=1}^{p-1}\sum_{k=2}^p (-1)^k\lambda_{\hk,x_i^q}\,.
\end{equation}
Next we use the fact that $u\in D^P$ to show that this coefficient is indeed 0. Since $\alpha\in P$, we, in particular, have
$u=\alpha^i \cdot u$, for every $i\in\{1,\ldots,p-1\}$. So let $i\in\{1,\ldots,p-1\}$, and
let $x\in\{x_1^q,\ldots,x_p^q\}$. Suppose that $b\in\mathcal{B}$ is such that $\he_{i+1}\w e_x$ occurs in 
$\alpha^i \cdot b$ with non-zero coefficient. 
Then from Lemma~\ref{lemma acts alpha} we deduce that $b=\he_{\alpha^{-i}(1)}\w e_x$. Moreover, we have
$$\alpha^i\cdot (\he_{\alpha^{-i}(1)}\w e_x)=(e_{i+2}-e_{i+1})\w\cdots\w (e_p-e_{i+1})\w (e_2-e_{i+1})\w\cdots\w (e_i-e_{i+1})\w (e_x-e_{i+1})\,.$$
Thus, by Lemma~\ref{lemma s(k)}, the coefficient at $\he_{i+1}\w e_x$ in $\alpha^i\cdot (\he_{\alpha^{-i}(1)}\w e_x)$ equals
$(-1)^{s(i+2,p)(i-1)}=1$. Letting $i$ vary over $\{1,\ldots,p-1\}$ and comparing the coefficient at 
$\he_{i+1}\w e_x$ in $u$ and in $\alpha^i \cdot u$, we deduce that $\lambda_{\hk,x}=\lambda_{\widehat{p-k+2},x}$, for
$k\in\{2,\ldots, (p+1)/2\}$ and every $x\in\{x_1^q,\ldots,x_p^q\}$. Since $k$ is even if and only if $p-k+2$ is odd, we 
conclude that the right-hand side of (\ref{eqn e in Tru}) is 0, as claimed. This completes the proof in case 1.

\medskip

\underline{Case 2:} $\alpha\notin P$, so that $\{1,\alpha,\alpha^2,\ldots,\alpha^{p-1}\}$ is a set of representatives for the cosets of $P$ in $E$, and we get 
$\Tr_P^E(u)=u+\alpha \cdot u+\cdots +\alpha^{p-1}\cdot  u$.
We determine the coefficient at $e$ in $\Tr_P^E(u)=u+\alpha \cdot u+\cdots +\alpha^{p-1} \cdot u$.
Let $i\in\{0,\ldots,p-1\}$, and let $b\in\mathcal{B}$ be such that $e$ occurs
in $\alpha^i \cdot b$ with non-zero coefficient. If $i=0$ then $b=e=\alpha^0 \cdot e$. So let $i\geq 1$. Then, by Lemma~\ref{lemma acts alpha},
we either have $b=e$, or $b=\he_{\alpha^{-i}(1)}\w e_x$, for some $x\in\{p+1,\ldots,n-1\}$. Moreover, in the latter case,
$$\alpha^i\cdot b=(e_{i+2}-e_{i+1})\w (e_{i+3}-e_{i+1})\w\cdots \w (e_p-e_{i+1})\w (e_2-e_{i+1})\w\cdots\w (e_i-e_{i+1})\w (e_x-e_{i+1})\,.$$
So the coefficient at $e$ in $\alpha^i\cdot (\he_{\alpha^{-i}(1)}\w e_x)$ equals
$$(-1)^{s(i+2,p)(i-1)+s(i+2,p)+1}=\begin{cases}
1&\text{ if } 2\nmid i\,,\\
-1&\text{ if } 2\mid i\,.
\end{cases}$$
Since $i$ is even if and only if $\alpha^{-i}(1)$ is even, we deduce from this that the coefficient at $e$ in $u+\alpha \cdot  u+\cdots +\alpha^{p-1}\cdot u$ equals
\begin{equation}\label{eqn e in Tru a}
p\lambda_e+\sum_{x=p+1}^{n-1}\sum_{k=2}^p(-1)^{k+1}\lambda_{\hk,x}=\sum_{x=p+1}^{n-1}\sum_{k=2}^p(-1)^{k+1}\lambda_{\hk,x}\,.
\end{equation}
To show that this coefficient is 0, we again exploit the fact that $u\in D^P$. In fact, we shall show that
\begin{equation}\label{eqn sum 0}
\mathop{\sum_{x\in\supp(E_{p^l,j_l})}}_{x<n} \sum_{k=2}^p (-1)^{k+1} \lambda_{\hk,x}=0\,,
\end{equation}
for every $l\in\{2,\ldots, t\}$ and $1\leq j_l\leq m_l$. For each such $l$ and $j_l$, there is, by Lemma~\ref{lemma max in E},
some element $\sigma(l,j_l)\in P$ such that $\supp(E_{p^l,j_l})\subseteq \supp(\sigma(l,j_l))\subseteq \{p+1,\ldots,n\}$. 
Fixing $l$ and $j_l$, we write
$$\sigma:=\sigma(l,j_l)=(x_1^1,\ldots,x_p^1)\cdots (x_1^q,\ldots, x_p^q)\,,$$
for some $q\geq |E_{p^l,j_l}|/p$ and $\supp(\sigma)=\{x_i^j: 1\leq i\leq p,\, 1\leq j\leq q\}$.

\smallskip

\underline{Case 2.1:} $n\notin\supp(\sigma)$, or equivalently, $\supp(\sigma)\cap \supp(E_{p^t,m_t})=\emptyset$.
Let $x\in\supp(\sigma)$, let $k\in\{2,\ldots,p\}$, and let $b\in \mathcal{B}$ be such that $\he_k\w e_{x}$ occurs
in $\sigma \cdot b$ with non-zero coefficient. This forces $b=\he_k\w e_{\sigma^{-1}(x)}$, and
$\sigma\cdot (\he_k\w e_{\sigma^{-1}(x)})=\he_k\w e_x$. Thus, $\lambda_{\hk,x}=\lambda_{\hk,\sigma^{-1}(x)}$.
This shows that $\lambda_{\hk,x_1^s}=\lambda_{\hk,x_i^s}$, for all $i\in\{1,\ldots,p\}$ and $s\in\{1,\ldots,q\}$.
By rearranging commuting $p$-cycles in $\sigma$, we may assume that there is some
$1\leq q_0\leq q$ such that $\supp(E_{p^l,j_l})=\{x_i^s: 1\leq i\leq p,\, 1\leq s\leq q_0\}$. Then
\begin{equation}\label{eqn sum 1}
\sum_{x\in\supp(E_{p^l,j_l})}\sum_{k=2}^p(-1)^{k+1} \lambda_{\hk,x}=\sum_{i=1}^p\sum_{s=1}^{q_0}\sum_{k=2}^p(-1)^{k+1} \lambda_{\hk,x_i^s}=p\sum_{s=1}^{q_0}\sum_{k=2}^p (-1)^{k+1}\lambda_{\hk,x_1^s}=0\,,
\end{equation}
as desired.

\smallskip

\underline{Case 2.2:} $n\in\supp(\sigma)$.
Then we may suppose that $x_p^q=n$. If $(l,j_l)\neq (t,m_t)$, then we may
further suppose that there is some $1\leq q_1<q$ such that 
$\supp(E_{p^l,j_l})=\{x_i^s: 1\leq i\leq p,\, 1\leq s\leq q_1\}$. 
By Lemma~\ref{lemma sigma}(b), we then get
\begin{equation}\label{eqn sum 4}
\sum_{x\in\supp(E_{p^l,j_l})}\sum_{k=2}^p (-1)^{k+1}\lambda_{\hk,x}=\sum_{i=1}^p\sum_{s=1}^{q_1}\sum_{k=2}^p(-1)^{k+1} \lambda_{\hk,x_i^s}=p\sum_{s=1}^{q_1}\sum_{k=2}^p(-1)^{k+1}  \lambda_{\hk,x_1^s}=0\,.
\end{equation}

If $(l,j_l)=(t,m_t)$ then we may suppose that there is 
$1\leq q_2\leq q$ such that $\supp(E_{p^t,m_t})=\{x_i^s: 1\leq i\leq p,\, q_2\leq s\leq q\}$.
In this case, Lemma~\ref{lemma sigma} gives
\begin{align*}
\mathop{\sum_{x\in\supp(E_{p^t,m_t})}}_{x<n}\sum_{k=2}^p(-1)^{k+1} \lambda_{\hk,x}&=\sum_{i=1}^p\sum_{s=q_2}^{q-1}\sum_{k=2}^p(-1)^{k+1}\lambda_{\hk,x_i^s}+\sum_{i=1}^{p-1}\sum_{k=2}^p(-1)^{k+1}\lambda_{\hk,x_i^q}\\
&=p\sum_{s=q_2}^{q-1}\sum_{k=2}^p(-1)^{k+1} \lambda_{\hk,x_1^s}-\sum_{i=1}^{p-1}\sum_{k=2}^p(-1)^{k}\lambda_{\hk,x_i^q}=0\,.
\end{align*}

To summarize, we have now verified equation (\ref{eqn sum 0}), which together with (\ref{eqn e in Tru}) shows that the coefficient
at $e$ in $\Tr_P^E(u)$ is 0. This now completes the proof in case 2 and, thus, of the lemma.
\end{proof}

As a direct consequence of Lemma~\ref{lemma TrP},  Corollary~\ref{cor e}, Proposition~\ref{prop Brauer}, and \cite[(1.3)]{Br} we thus have proved the following

\begin{prop}\label{prop Brauer D(E)}
Let $n\in\NN$ be such that $n=p+\sum_{i=2}^tm_ip^i$, for some $t\geq 2$, $m_2,\ldots,m_t\in\NN_0$
with $m_t\neq 0$. Let further $D:=D^{(n-p+1,1^{p-1})}$, and let $Q\leq \mathfrak{S}_n$ be a vertex of $D$. 
Then $D(E(1,m_2,\ldots,m_t))\neq \{0\}$; in particular, $E(1,m_2,\ldots,m_t)\leq_{\mathfrak{S}_n} Q$.
\end{prop}

\begin{rem}\label{rem Brauer D(E)}
Again consider the $p$-adic expansion $n=p+\sum_{i=2}^rn_ip^i$, where $r\geq 2$ and $n_r\neq 0$.
Note that Proposition~\ref{prop Brauer D(E)}, in particular, holds for $t=r$ and $m_1=1,\ldots, m_r=n_r$.
Thus the elementary abelian subgroup $E(1,n_2,\ldots,n_r)\leq P_n$ is $\mathfrak{S}_n$-conjugate to a subgroup
of every vertex of $D$. This settles item (ii) at the beginning of this section and  completes the proof of Theorem~\ref{thm main}.
\end{rem}


\bigskip

\noindent
{\sc S.D.: Department of Mathematics,
University of Kaiserslautern,\\
P.O. Box 3049,
67653 Kaiserslautern,
Germany}\\
{\sf danz@mathematik.uni-kl.de}

\bigskip

\noindent
{\sc E.G.: Department of Mathematics, Royal Holloway University of London,\\
Egham TW20
0EX, United Kingdom}\\
{\sf Eugenio.Giannelli.2011@live.rhul.ac.uk}

\end{document}